\documentclass{siamart171218}
\usepackage{amsfonts}
\usepackage{graphicx}
\usepackage{epstopdf}
\ifpdf
\DeclareGraphicsExtensions{.eps,.pdf,.png,.jpg}
\else
\DeclareGraphicsExtensions{.eps}
\fi

\newsiamremark{remark}{Remark}
\newsiamremark{hypothesis}{Hypothesis}
\crefname{hypothesis}{Hypothesis}{Hypotheses}
\newsiamthm{claim}{Claim}

\headers{A Cucker-Smale model with decentralized formation control}{Y.-P. Choi, D. Kalise, J. Peszek and A. A. Peters}

\title{A collisionless singular Cucker-Smale model with decentralized formation control\thanks{Submitted to the editors \today.}}

\author{Young-Pil~Choi\thanks{Y-P. Choi is with the  Department of Mathematics,
		Yonsei University, 50 Yonsei-Ro, Seodaemun-Gu, Seoul 03722, Republic of Korea
		(\email{ypchoi@yonsei.ac.kr}).}
	\and Dante~Kalise\thanks{D. Kalise is with the School of Mathematical Sciences, University of Nottingham, University Park, Nottingham NG7 2QL, United Kingdom (\email{dante.kalise@nottingham.ac.uk}).}
	\and Jan~Peszek\thanks{J. Peszek is with the Center for Scientific Computation and Mathematical Modeling (CSCAMM), University of Maryland, College Park, MD 20742-4015, USA, and with the Institute of Mathematics of the Polish Academy of Sciences, ul. \'Sniadeckich 8
		00-656 Warszawa, Poland (\email{j.peszek@mimuw.edu.pl}).}
	\and Andr\'es~A.~Peters\thanks{A. Peters is with the Department of Electricity, Universidad Tecnológica Metropolitana, Av. Jose Pedro Alessandri 1242, 7800002, Santiago, Chile
		 (\email{a.petersr@utem.cl})}
}

\usepackage{amsopn}

\usepackage{amssymb,graphicx}
\usepackage{amsxtra}
\usepackage{indentfirst}
\usepackage{hyperref}
\usepackage{ifthen}
\usepackage{url}
\usepackage{colortbl}
\usepackage{color}
\definecolor{black}{rgb}{0.0, 0.0, 0.0}
\definecolor{red}{rgb}{0.82, 0.0, 0.0}

\def\({\begin{eqnarray}}
\def\){\end{eqnarray}}
\def\[{\begin{eqnarray*}}
\def\]{\end{eqnarray*}}
\newcommand{\R}{\mathbb R}

\newcommand{\bq}{\begin{equation}}
\newcommand{\eq}{\end{equation}}

\newcommand{\lt}{\left}
\newcommand{\rt}{\right}
\newcommand{\lal}{\langle}
\newcommand{\ral}{\rangle}
\newcommand{\pa}{\partial}

\begin{document}
%

%
%
%

\maketitle

\begin{abstract}
We address the design of decentralized feedback control laws inducing consensus and prescribed spatial patterns over a singular interacting particle system of Cucker-Smale type. The control design consists of a feedback term regulating the distance between each agent and pre-assigned subset of neighbours. Such a design represents a multidimensional extension of existing control laws for 1d platoon formation control. For the proposed controller we study consensus emergence, collision-avoidance and formation control features in terms of energy estimates for the closed-loop system. Numerical experiments in 1, 2 and 3 dimensions assess the different features of the proposed design.
\end{abstract}

\begin{keywords}
	Multi-agent systems, decentralized control, flocking, pattern formation.
\end{keywords}

\begin{AMS}
34H05, 93A15, 93C10 
\end{AMS}

\section{Introduction}
Multi-agent systems (MAS) provide a versatile frame\-work for modelling different challenges arising in Science and Engineering, such as collective animal and human behaviour \cite{sumpter,vicsek}, dynamic networks \cite{olfati}, and autonomous vehicles \cite{balch}, among many others. From a mathematical viewpoint, MAS are often modelled as large-scale dynamical systems where each agent is represented by a subset of states which are updated via ``physical'' interaction rules \cite{helbing,siads19} (attraction, repulsion, alignment, synchronization etc.), or by means of a control/game framework \cite{lasry,huang}. 

In this work, we are concerned with the design of dynamic interactions and external control laws for nonlinear MAS representing the physical motion of a swarm of agents. Our mathematical modelling is inspired by animal collective dynamics, where large populations of birds and fish normally exhibit self-organization behaviour such as flocking, swarming, milling, and alignment. Other remarkable examples are linked to pedestrian behaviour and to platoons of unmanned  aerial vehicles (UAVs). In particular, we are interested in prescribing nonlinear dynamics for the MAS leading to self-organized flocking and trajectories with collision-avoidance features, the latter being a fundamental aspect for a realistic model. The emergence of collisionless flocking behaviour, understood as a configuration in which agents travel with the same constant velocity, is already a complex dynamic equilibrium of interest on its own right.  However, a flocking regime does not provide a complete account of the spatial configuration of the swarm, thus limiting its practical interest for applications, such as pedestrian modelling and UAV control. Hence, it is desirable to endow the MAS dynamics with additional forcing terms which can also induce the formation of a given spatial configuration. In this paper, we propose a dynamical MAS model  including collisionless flocking, together with a control law inducing a given spatial configuration. The control action we propose is inspired by the literature concerning string stability for autonomous vehicles, and consists of a decentralized feedback law requiring a reduced amount of communication between agents.

Let us briefly review the technical aspects of our work and the related literature. The starting point for our model is the seminal paper by Cucker and Smale \cite{cs07}, where the authors propose a nonlinear second-order model for multi-agent flocking dynamics. This work has been later extended along different directions, including collision-avoidance features \cite{CD,cdong2,cdong3,achl,bf,Choi2017}, forcing terms and control \cite{cuckerhuepe,HaHaKim,cfpt,BFK,bbck,hk18}, formation control \cite{pkh10,perea}, leadership \cite{dalmao} and mean-field modelling of large-scale swarms \cite{ccp17,ncmb,acfk,CFRT,HaLiu,cch}. Starting from a Cucker-Smale type model, we study collision avoidance in the framework of \cite{cch14,CCMP,mp18,kp}, where singular interaction kernels have been proposed. Such interaction kernels blow-up whenever two agents are located at the same position, providing an adequate framework for the study of collision-less control laws. While these works discuss the emergence of flocking and collision-avoidance features, we focus on the design of additional external signals enforcing a given spatial configuration for the swarm of agents. 

In order to induce spatial configurations over the swarm, we design a decentralized controller \cite{Bakule,Jadbabaie}, similar to those that have been studied in the context of 1D vehicle platoons \cite{Peters,surveystringstab}. The control strategy uses a low number of new interactions with respect to the original Cucker-Smale model, in order to achieve arbitrary spatial configurations. We extend these ideas to formation control in arbitrary dimensions and illustrate the 2D and 3D cases with numerical examples. Our main result is the characterization of the set of initial configurations leading to consensus with a pre-specified spatial pattern under a given decentralized control law, together with a certification of collision-avoidance along the trajectory. Up to our knowledge, this is the first work achieving such a design with Cucker-Smale type dynamics under the action of a decentralized external signal. In applications, decentralized control schemes have a clear economical benefit, but they are not necessarily sacrificing performance, when compared to more complex solutions \cite{balch}. Additionally, in higher dimensions, self organizing agents and their control take relevance if the extra states are interpreted as a feature of the agent that is prone to be measured and/or used for synchronization \cite{StrogatzOscil}. 

The control of MAS, and in particular drone and/or robot swarms is an intensive research topic (see the recent survey \cite{surveymaformationcontrol}), with applications ranging from spacecraft formation \cite{ccb, perea}, robot self-organization \cite{ferrante}, surveillance \cite{surveysurveillance} and localization \cite{surveylocal}, to digital media arts \cite{Kim2018}. Along this line, in Section \ref{numerics} we present different numerical simulations for planar and spatial formations, mimicking the display of UAV swarms forming the Olympic Rings symbol, as in the Intel project shown at the opening ceremony of the XXIII Olympic Winter Games, 2018, in PyeongChang, South Korea \cite{olympic}. Full simulation videos can be accessed at \href{https://youtu.be/C7UDGRudsyA}{https://youtu.be/C7UDGRudsyA}.

The rest of the paper is structured as follows. In Section \ref{ps} we present our main system, which is a Cucker-Smale model with singular interactions and a decentralized feedback control. In Section \ref{apriori} we introduce a total energy functional $E(x,v)$ for our system and show that the total energy functional is not increasing in time. Section \ref{regularity} presents a result concerning the collision-avoidance behaviour of the controlled system, and Section \ref{flocking} is devoted to provide a flocking estimate showing the velocity alignment between individuals as time goes to infinity and the relative positions are uniformly bounded in time. In Section \ref{formation} we show a result regarding the formation control, to conclude with different numerical experiments in Section \ref{numerics}.

\section{Problem description}\label{ps}
We study the dynamics of an $n$-particle system in which the position and velocity of the $i-$th agent, denoted $x_i(t) \in \R^d$ and $v_i(t) \in \R^d$ respectively, evolves according to
\begin{align}\label{main_eq}
\begin{aligned}
	\frac{d x_i(t)}{dt} &= v_i(t),\quad i=1,\dots, n, \quad t > 0,\cr
	\frac{d v_i(t)}{dt} &= \frac{K}{n}\sum_{j=1}^n \psi(r_{ij}(t))(v_j(t) - v_i(t)) + Mu_i(t),
\end{aligned}
\end{align}
subject to the initial data 
\bq\label{ini_main_eq}
(x_i(0), v_i(0)) =: (x_i^0, v_i^0) \quad \mbox{for} \quad i =1,\dots,n.
\eq
Here, $r_{ij}$ denotes the Euclidean distance between $i$ and $j$-th individuals at time $t$, i.e., $r_{ij}(t):= |x_j(t) - x_i(t)|$,  and $K$ and $M$ are nonnegative constants. The first term on the right hand side of the equation for $v_i$ in \eqref{main_eq} represents a nonlocal velocity alignment force where the discrepancies between the velocities of the agents are averaged and weighted by the communication function $\psi$ in such a way that closer individuals have stronger influence than further ones. 

Let us discuss the choice of the interaction kernel $\psi(r)$. In the original Cucker-Smale model \cite{cs07}, the regular weight function $\psi(r) = 1/(1+r^2)^{\beta/2}$ is considered without a control term, i.e., the system \eqref{main_eq} with $M=0$. Depending on the exponent $\beta$ which determines the short/long-range interaction regimes, conditional/unconditional flocking estimates were obtained in \cite{cs07}, and later refined in \cite{CFRT,HaLiu}. Inspired by the recent work \cite{CCMP}, here we consider a singular influence function $\psi : (0, +\infty) \to (0, +\infty)$ given by
\[
\psi(r) = \frac{1}{r^\alpha}, \quad \alpha > 0.
\]
In \cite{CCMP}, a critical value for the exponent $\alpha$ leading to global regularity of solutions or finite-time collision between individuals is obtained. More precisely, if $\alpha \geq 1$ and the individuals are placed in different positions initially, then there is no collision between them when governed by the system \eqref{main_eq} with $M=0$, and this immediately provides the global existence and uniqueness of solutions to the system \eqref{main_eq} in the absence of control. On the other hand, as shown in \cite{jpe}, for $\alpha\in(0,1)$ the particles may collide and even stick together in a finite time.

The second term on the right hand side of the equation for $v_i$ in \eqref{main_eq} is an external control signal $u := (u_1, \cdots, u_n)$. The design of external control actions acting over the free dynamics on a prescribed way constitutes a challenging topic on its own right. In this context, we focus on the design and analysis of a control signal able to induce a desired spatial configuration for the swarm. For this, we make the following choice:
$$\begin{aligned}
u_1 &= -\phi(|x_1 - x_2 - z_1|^2)(x_1 - x_2 - z_1),\cr
u_n &= \phi(|x_{n-1} - x_n - z_{n-1}|^2)(x_{n-1} - x_n - z_{n-1}),\cr
u_i &= \phi(|x_{i-1} - x_i - z_{i-1}|^2)(x_{i-1} - x_i - z_{i-1}) \\
& \quad-\phi(|x_i - x_{i+1} - z_i|^2)(x_i - x_{i+1} - z_i),
\end{aligned}$$
for $i \in \{2,\cdots, n-1\}$, where $z_i \in \R^d$, $i=1,\cdots, n-1$ determines a relative positioning of the agents based on the desired spatial configuration, and $\phi(r)$ is a regular communication weight of the form
\[
\phi(r) = \frac{1}{(1 + r)^\beta}, \quad \beta > 0.
\]

While the choice for the controller may seem arbitrary, it has two appealing features. First, the signal is a feedback control, as its computation is based solely on the current state of the system and predefined parameters. Second, the control signal is decentralized: its update for a single agent only requires information concerning the state of a reduced number of individuals (in our case only two predefined agents) rather than the knowledge of the full swarm state. Overall, the proposed design achieves a prescribed formation in a robust and economic way. In the following sections, we carry a thorough analysis and computational validation of the different  properties of the resulting controlled dynamics.

For notational simplicity, we set
\[
\|w\|^2 := \sum_{i=1}^n|w_i|^2 \quad \mbox{and} \quad w_c := \frac1n \sum_{i=1}^n w_i,
\]
for $w = (w_1,\cdots,w_n) \in \R^{d n}$.
%
%
%
%
\section{Preliminaries: A priori estimates}\label{apriori}
In this section, we present a priori estimates of the average quantities and total energy, which will be significantly used for the flocking estimate. 

We first set the total energy functional
\begin{align*}
E_1(v) + E_2(x) &:= \frac12 \|v\|^2 + \frac M2\sum_{i=2}^n\int_0^{|x_{i-1} - x_i - z_{i-1}|^2}\phi(r)\,dr,
\end{align*}
and its dissipation rate
\[
D(x,v) := \frac{K}{2n}\sum_{i,j=1}^n \psi(r_{ij})|v_i - v_j|^2. 
\]

\begin{lemma}\label{lem_energy}Let $T>0$ and $\{(x_i,v_i)\}_{i=1}^n$ be a smooth solution to the system \eqref{main_eq} in the time interval $[0,T]$. Then we have the explicit from of the averages:
\[
x_c(t) = x_c(0) + v_c(0)t, \qquad v_c(t) = v_c(0),
\]
and the non-increasing total energy estimate:
\[
\frac{d}{dt}E_1(v(t)) + \frac{d}{dt}E_2(x(t)) + D(x(t),v(t)) = 0,
\]
for $0 \leq t \leq T$.
\end{lemma}
\begin{proof}
By the definition of the average quantities together with the fact that the symmetry of the communication weight function $\psi$ and $u_c = 0$, it is simple to get
\[
\frac{d x_c(t)}{dt} = v_c(t), \quad  \mbox{and} \quad \frac{d v_c(t)}{dt} = 0.
\]
We next estimate the kinetic energy. A straightforward computation yields
\begin{align}\label{est_v0}
\begin{aligned}
	\frac12 \frac{d}{dt}\|v\|^2 &= \sum_{i=1}^n \lt\lal v_i, \frac{d v_i}{dt}\rt\ral\cr
	&= \sum_{i=1}^n \lt\lal v_i, \frac{K}{n}\sum_{j=1}^n \psi(r_{ij})(v_j - v_i) + M u_i\rt\ral\cr
	&=: I_1 + M I_2,
\end{aligned}
\end{align}
where $\lal \cdot, \cdot\ral$ denotes the standard inner product in $\R^d$. Here, by substituting indices $i$ and $j$, and using the symmetry of the weight function $\psi$, $I_1$ can be easily estimated as 
\bq\label{est_i1}
I_1 = -\frac{K}{2n}\sum_{i,j=1}^n \psi(r_{ij})|v_i - v_j|^2 = - D(x,v).
\eq
For the estimate of $I_2$, we obtain
\begin{align}\label{est_i2}
\begin{aligned}
	I_2 &= \lal u_1, v_1\ral  + \sum_{i=2}^{n-1}\lal u_i, v_i \ral + \lal u_n, v_n \ral\cr
	&= \lal u_1, v_1 \ral + \sum_{i=2}^{n-1}\phi(|x_{i-1} - x_i - z_{i-1}|^2)\lal x_{i-1} - x_i - z_{i-1}, v_i\ral \cr
	&\quad - \sum_{i=2}^{n-1}\phi(|x_i - x_{i+1} - z_i|^2)\lal x_i - x_{i+1} - z_i, v_i \ral +\lal  u_n, v_n \ral\cr
	&=\lal u_1, v_1 \ral + \sum_{i=2}^{n-1}\phi(|x_{i-1} - x_i - z_{i-1}|^2)\lal x_{i-1} - x_i - z_{i-1}, v_i \ral \cr
	&\quad - \sum_{i=3}^{n}\phi(|x_{i-1} - x_i - z_{i-1}|^2)\lal x_{i-1} - x_i - z_{i-1}, v_{i-1}\ral + \lal u_n, v_n \ral \cr
	&=-\sum_{i=2}^n\phi(|x_{i-1} - x_i - z_{i-1}|^2)\lal x_{i-1} - x_i - z_{i-1}, v_{i-1} - v_i \ral\cr
	&= -\frac 12\frac{d}{dt}\sum_{i=2}^n\int_0^{|x_{i-1} - x_i - z_{i-1}|^2}\phi(r)\,dr.
\end{aligned}
\end{align}
Combining the estimates \eqref{est_v0}, \eqref{est_i1}, and \eqref{est_i2}, we conclude the desired result. 
\end{proof}
\begin{remark}\label{rmk_energy} Since the velocity average is conserved in time, i.e., $v_c'(t) = 0$, the time derivative of the kinetic energy can be rewritten as
\[
\frac{d}{dt}E_1(v(t)) = \frac{1}{4n}\frac{d}{dt}\sum_{i,j=1}^n|v_i(t) - v_j(t)|^2.
\]
This yields that we can rewrite the total energy estimate appeared in Lemma \ref{lem_energy} as
\[
\frac{d}{dt}\lt(\frac{1}{4n} \sum_{i,j=1}^n|v_i(t) - v_j(t)|^2 + E_2(x(t)) \rt) + D(x(t),v(t)) = 0.
\]
Then we obtain the following uniform-in-time estimate: 
\[
E(x(t),v(t)):=\frac{1}{4n} \sum_{i,j=1}^n|v_i(t) - v_j(t)|^2 + E_2(x(t)) \leq \frac{1}{4n} \sum_{i,j=1}^n|v_i^0 - v_j^0|^2 + E_2(x(0))=:E_0.
\]
In particular, we obtain
\[
\max_{1 \leq i,j \leq n}|v_i(t) - v_j(t)| \leq \sqrt{ \sum_{i,j=1}^n|v_i(t) - v_j(t)|^2} \leq 2\sqrt{nE_0}.
\]
\end{remark}
\begin{remark} If $v_c(0) = 0$, then $v_c(t) \equiv 0$ for all $t \geq 0$, and this subsequently gives
\[
E_1(t) = \frac{1}{4n} \sum_{i,j=1}^n|v_i(t) - v_j(t)|^2 \quad \mbox{for all} \quad t \geq 0.
\]
Without loss of generality, throughout this work, we may assume that $v_c(0) = 0$ and $x_c(0)=0$. If necessary, we may consider the translation frame $(x_i, v_i) \to (x_i - v_c t, v_i)$. This implies 
\begin{align}\label{zero}
	x_c(t)=0 \quad \mbox{and} \quad v_c(t) = 0, 
\end{align}
for all $t \geq 0$.

\end{remark}

%
%
%
%
\section{Non-collisional behavior: Global regularity}\label{regularity}
In this section, we study the non-collisional behavior of the system \eqref{main_eq} based on the work \cite{CCMP}, when $\alpha \geq 1$. This fact together with Cauchy-Lipschitz theory implies the global existence and uniqueness of smooth solutions to the system \eqref{main_eq}-\eqref{ini_main_eq}.

\begin{theorem}\label{thm_ext} Suppose that $\alpha \geq 1$ and the initial data $x_0$ satisfy 
\[
\min_{1 \leq i \neq j \leq n} r_{ij}(0) > 0.
\]
Then there exists a global smooth solution to the system \eqref{main_eq}-\eqref{ini_main_eq} satisfying 
\[
\min_{1 \leq i \neq j \leq n} r_{ij}(t) > 0 \quad \mbox{for} \quad t \geq 0.
\]
\end{theorem}
\begin{remark} In \cite{CD}, the repulsive forcing term is added to the original Cucker-Smale model to avoid collisions between particles. In contrast, we extract a repulsive forcing effect by taking into account the singular weights in the velocity alignment force.
\end{remark}
\begin{remark}From the perspective of applications, the existence of a global-in-time minimal distance between the agents is often desirable. Theorem \ref{thm_ext} does not ensure its existence, since the minimum inter particle distance $\min_{1 \leq i \neq j \leq n} r_{ij}(t)$ can be equal to zero when $T = +\infty$. To circumvent this issue one may consider the CS model with an expanded range of singularity $\psi_\delta(s) = \psi(s-\delta)$ for $s>\delta$. It was shown in \cite{CCMP} that the Cucker-Smale model with weight $\psi_\delta$ and singularity $\alpha\geq 2$ admits a global-in-time minimal distance between the particles, $\min_{1 \leq i \neq j \leq n} r_{ij}(t)\geq \delta$. System \eqref{main_eq} with weight $\psi_\delta$ is expected to exhibit similar behaviour. Evidently, we should set $|z_i|>\delta$ for all $i$ in the control signal $u$, whenever $\delta>0$.
\end{remark}
\begin{proof}[Proof of Theorem \ref{thm_ext}] We first fix a time $T>0$, and we will show that there is no collision between particles until that time. Since there are no particles colliding at the initial time, there exists a $t_0 \in (0,T]$ such that the smooth solution uniquely exists until that time. Let us assume that $t_0 < T$, i.e., $t_0$ is the first time of collision of any particles. We then set $[l]$ to be the set of all indices $j \in \{1,\cdots,n\}$ where the $j$-th particle collides with the $l$-th particle, i.e., $r_{jl}(t) \to 0$ as $t \to t_0$ for all $j \in [l]$ and $r_{jl}(t) \geq \delta > 0$ in $[0,t_0)$ for all $j \notin [l]$ and some positive constant $\delta > 0$. Due to our assumption, $|[l]| > 1$. We set
\[
\|x\|_{[l]}(t) := \sqrt{\sum_{i,j \in [l]}|x_i(t) - x_j(t)|^2}, \quad \|v\|_{[l]}(t) := \sqrt{\sum_{i,j \in [l]}|v_i(t) - v_j(t)|^2}, 
\]
and
\[ 
\|u\|_{[l]}(t) := \sqrt{\sum_{i,j \in [l]}|u_i(t) - u_j(t)|^2}.
\]
Note that by definition $\|x\|_{[l]}(t) \to 0$ as $t \to t_0-$. Our goal is to show that this cannot happen, thus $t_0 \geq T$. We now analyze the evolution of $\|x\|_{[l]}(t)$ and $\|v\|_{[l]}(t)$. We can first easily estimate
\bq\label{est_x}
\frac{d}{dt}\|x\|_{[l]}^2 \leq 2\|x\|_{[l]}\|v\|_{[l]}, \quad \mbox{i.e.,} \quad \lt|\frac{d}{dt}\|x\|_{[l]} \rt| \leq \|v\|_{[l]}.
\eq
We now estimate $\|v\|_{[l]}(t)$ as 
$$\begin{aligned}
\frac{d}{dt}\|v\|_{[l]}^2 &= 2\sum_{i,j\in[l]}\lt\lal v_i-v_j, \left[\frac{K}{n} \sum_{k=1}^n\psi(r_{ki})(v_k-v_i)- \frac{K}{n}\sum_{k=1}^n\psi(r_{kj})(v_k-v_j)\right]\rt\ral\\
&\quad + 2M\sum_{i,j\in[l]}\lal v_i-v_j, u_i - u_j\ral\cr
&=\frac{2K}{n}\left(\sum_{i,j,k\in[l]}+\sum_{\stackrel{i,j\in[l]}{k\notin[l]}}\right) \left[\psi(r_{ik})\lal v_i-v_j, v_k-v_i\ral-\psi(r_{jk})\lal v_i-v_j, v_k-v_j\ral\right]\cr
&\quad + 2M\sum_{i,j\in[l]}\lal v_i-v_j, u_i - u_j\ral\cr
&=: I_1 + I_2 + I_3.
\end{aligned}$$
$\diamond$ Estimate of $I_1$: By exchanging the indices $i$ and $k$, we find
$$\begin{aligned}
&\frac{2K}{n}\sum_{i,j,k\in[l]}\psi(r_{ki})\lal v_i-v_j, v_k-v_i\ral \\
&\quad =\frac{K}{n}\sum_{i,j,k\in[l]}\psi(r_{ki})\lal v_i-v_j , v_k-v_i\ral +\frac{K}{n}\sum_{i,j,k\in[l]}\psi(r_{ki})\lal v_k-v_j, v_i-v_k\ral\\
&\quad = \frac{K}{n}\sum_{i,j,k\in[l]}\psi(r_{ki})\lal v_i-v_k , v_k-v_i\ral\\
&\quad =-\frac{K|[l]|}{n}\sum_{i,j\in[l]}\psi(|x_i-x_j|)|v_i-v_j|^2.
\end{aligned}$$
In a similar fashion as the above, we get
\[
- \frac{2}{n}\sum_{i,j,k\in[l]}^n\psi(r_{kj})\lal v_i-v_j , v_k-v_j\ral = -\frac{|[l]|}{n}\sum_{i,j\in[l]}\psi(r_{ij})|v_i-v_j|^2.
\]
This together with the fact that $|x_i-x_j|\leq  \|x\|_{[l]}$ for all $i,j\in[l]$ yields
\begin{align*}
I_1 \leq -\frac{2K|[l]|}{n}\sum_{i,j\in[l]}\psi(r_{ij})|v_i-v_j|^2 \leq -2C_1\psi(\|x\|_{[l]})\|v\|_{[l]}^2,
\end{align*}
where $C_1 :=K |[l]|/n$. \newline

\noindent $\diamond$ Estimate of $I_2$: A straightforward computation gives
\begin{align*}
I_2 &= \frac{2K}{n}\sum_{\stackrel{i,j\in[l]}{k\notin[l]}}\psi(r_{ik})\lal v_i-v_j, v_j-v_i\ral + \frac{2K}{n}\sum_{\stackrel{i,j\in[l]}{k\notin[l]}}\left(\psi(r_{ik})-\psi(r_{jk})\right)\lal v_i-v_j, v_k-v_j\ral \\
&= -\frac{2K}{n}\sum_{\stackrel{i,j\in[l]}{k\notin[l]}} \psi(r_{ik})|v_i-v_j|^2 + \frac{2K}{n}\sum_{\stackrel{i,j\in[l]}{k\notin[l]}}\left(\psi(r_{ik})-\psi(r_{jk})\right)\lal v_i-v_j , v_k-v_j\ral \\
&\leq \frac{2KL_\delta}{n}\sum_{\stackrel{i,j\in[l]}{k\notin[l]}}|\lal v_i-v_j, v_k-v_j\ral||x_i-x_j|,
\end{align*}
where $L_\delta$ is the Lipschitz constant of $\psi$ in the interval $(\delta,\infty)$. On the other hand, it follows from Remark \ref{rmk_energy} that
\[
\max_{1 \leq k,j \leq n}|v_k - v_j| \leq 2\sqrt{nE_0}.
\]
Thus we can estimate $I_2$ as
$$\begin{aligned}
I_2 &\leq \frac{4K\sqrt{E_0}L_\delta}{\sqrt{n}}\sum_{\stackrel{i,j\in[l]}{k\notin[l]}}|v_i - v_j||x_i - x_j| \cr
&= \frac{4K\sqrt{E_0}L_\delta(n - |[l]|)}{\sqrt{n}}\sum_{i,j \in [l]}|v_i - v_j||x_i - x_j| \cr
&\leq 2C_2 \|v\|_{[l]}\|x\|_{[l]},
\end{aligned}$$
where $C_2 > 0$ is given by
\[
C_2 := \frac{2K\sqrt{E_0}L_\delta(n - |[l]|)}{\sqrt{n}}.
\]

\noindent $\diamond$ Estimate of $I_3$: We first notice that for any $i \in \{1,\dots,n\}$
$$\begin{aligned}
|u_i| &\leq \sum_{k=2}^n|x_{k-1} - x_k| + \sum_{k=2}^n |z_{k-1}|\cr
&\leq \sqrt{n}\sqrt{\sum_{k=2}^n|x_{k-1} - x_k|^2} + \sum_{k=2}^n |z_{k-1}|\cr
&\leq \sqrt{n}\Gamma(x) + |z|_{\ell^1},
\end{aligned}$$
where
\[
\Gamma(x) := \sqrt{\sum_{i,j=1}^n|x_i - x_j|^2} \quad \mbox{and} \quad |z|_{\ell^1}:= \sum_{k=2}^n |z_{k-1}|.
\]
On the other hand, we get
\[
\frac{d}{dt}\Gamma(x(t))^2 \leq 2\Gamma(x(t)) \Lambda(v(t)) \quad \mbox{and} \quad \frac{d}{dt}\Gamma(x(t)) \leq  \Lambda(v(t)),
\]
where
\[
\Lambda(v(t)) := \sqrt{\sum_{i,j=1}^n|v_i(t) - v_j(t)|^2}.
\]
Since $\Lambda(v(t)) \leq 2\sqrt{nE_0}$ for all $t \geq 0$, this yields
\[
\Gamma(x(t)) \leq \Gamma(x(0)) +2\sqrt{nE_0}T.
\]
Thus we find
\[
|u_i(t)| \leq \sqrt{n}\Gamma(x(t)) + |z|_{\ell^1} \leq \sqrt{n}\lt(\Gamma(x(0)) + 2\sqrt{nE_0}T \rt) + |z|_{\ell^1} =: C_3>0,
\]
and furthermore we obtain
\[
\|u\|_{[l]} \leq \sqrt{2\sum_{i,j \in [l]}(|u_i|^2 + |u_j|^2) } \leq 2C_3|[l]|.
\]
Hence we have
\[
I_3 \leq 2M\|v\|_{[l]}\|u\|_{[l]} \leq 4C_3|[l]|M\|v\|_{[l]} =: 2C_4\|v\|_{[l]}.
\]
Combining all of the above estimates, we find
\[
\frac{d}{dt}\|v\|_{[l]}^2\leq -2C_1\psi\lt(\|x\|_{[l]}\rt)\|v\|_{[l]}^2 + 2C_2\|v\|_{[l]}\|x\|_{[l]} + 2C_4\|v\|_{[l]}.
\]
Since 
\[
\|x\|_{[l]}(t) \leq \Gamma(x(t)) \leq \Gamma(x(0)) + 2\sqrt{nE_0}T,
\]
if we set
\[
C_5 := C_2\lt(\Gamma(x(0)) + 2\sqrt{nE_0}T\rt) + C_4,
\]
then we get
\[
\frac{d}{dt}\|v\|_{[l]}^2\leq -2C_1\psi\lt(\|x\|_{[l]}\rt)\|v\|_{[l]}^2 + 2C_5\|v\|_{[l]},
\]
i.e.,
\[
\frac{d}{dt}\|v\|_{[l]} \leq -C_1\psi\lt(\|x\|_{[l]}\rt)\|v\|_{[l]} + C_5, \quad a.e. \mbox{ on } (s,t_0)
\]
with $0 \leq s < t_0$. We now apply Gronwall's inequality on the time interval $(s,t_0)$ to obtain
\begin{align}\label{est_v}
\begin{aligned}
	\|v\|_{[l]}(t) &\leq\|v\|_{[l]}(s)e^{-C_1\int_s^t \psi(\|x\|_{[l]}(\tau))\,d\tau} + C_5\int_s^t e^{-C_1\int_\tau^t \psi(\|x\|_{[l]}(\sigma))\,d\sigma} d\tau\cr
	&\leq C_5e^{-C_1\int_s^t \psi(\|x\|_{[l]}(\tau))\,d\tau} + C_5\int_s^t e^{-C_1\int_\tau^t \psi(\|x\|_{[l]}(\sigma))\,d\sigma} d\tau,
\end{aligned}
\end{align}
due to $\|v\|_{[l]}(t) \leq \Lambda(v(t)) \leq C_3 < C_4 < C_5$. Let us denote by $\Psi$ the primitive of $\psi$. Then we find from \eqref{est_x} that 
\begin{align*}
\left|\Psi(\|x\|_{[l]}(t))\right| &= \left|\int_s^t\frac{d}{dt}\Psi(\|x\|_{[l]}(\tau))\,d\tau
+\Psi(\|x\|_{[l]}(s))\right|\\
&=\left|\int_s^t\psi(\|x\|_{[l]}(\tau))\left(\frac{d}{dt}\|x\|_{[l]}\right)(\tau)\,d\tau
+ \Psi(\|x\|_{[l]}(s))\right|\\
&\leq \int_s^t\psi(\|x\|_{[l]}(\tau))\|v\|_{[l]}(\tau)\,d\tau +|\Psi(\|x\|_{[l]}(s))|.
\end{align*}
We now estimate
\[
J(t,s) := \int_s^t\psi(\|x\|_{[l]}(\tau))\|v\|_{[l]}(\tau)\,d\tau.
\]
Note that if $J$ is bounded from above by some constant $J_* > 0$, then we have
\begin{align*}
\left|\Psi(\|x\|_{[l]}(t))\right| \leq J_* +\Psi(\|x\|_{[l]}(s))|.
\end{align*}
This leads to a contradiction since the right hand side of the above inequality is bounded, however $\left|\Psi(\|x\|_{[l]}(t))\right| \to +\infty$ as $s<t \to t_0-$, and this gives $|[l]| = 0$, i.e., there is no collision between particles until time $T>0$. Thus in the rest of proof, we show the boundedness of $J$.  For notational simplicity, we set
\[
b(t,s) := \exp\lt(-C_1\int_s^t \psi(\|x\|_{[l]}(\tau))\,d\tau\rt).
\]
Then it follows from \eqref{est_v} that
\[
J(t,s) \leq \int_s^t\psi(\|x\|_{[l]}(\tau))\lt(b(\tau,s) + \int_s^\tau b(\tau,\sigma)\,d\sigma \rt)\,d\tau =: J_1(t,s) + J_2(t,s).
\]
Note that $b(t,s)$ has the following properties:
\[
\pa_t b(t,s) = -C_1\psi(\|x\|_{[l]}(t))b(t,s) \quad \mbox{and} \quad b(t,\tau)b(\tau,s) = b(t,s) \quad\mbox{for} \quad s \leq \tau \leq t.
\]
By using these properties, we estimate $J_i(t,s),i=1,2$ as
$$\begin{aligned}
J_1(t,s) &= -\frac{C_5}{C_1}\int_s^t \pa_\tau b(\tau,s)\,d\tau = -\frac{C_5}{C_1}\lt(b(t,s) - b(s,s) \rt) = \frac{C_5}{C_1}\lt(1 - b(t,s) \rt) \leq \frac{C_5}{C_1},\cr
J_2(t,s) &= C_5\int_s^t \psi(\|x\|_{[l]}(\tau))\int_s^\tau \frac{b(\tau,s)}{b(\sigma,s)}\,d\sigma d\tau \cr
&= C_5\int_s^t \psi(\|x\|_{[l]}(\tau))b(\tau,s)\lt(\int_s^\tau \frac{1}{b(\sigma,s)}\,d\sigma \rt)d\tau\cr
&= - \frac{C_5}{C_1}\int_s^t \pa_\tau b(\tau,s)\lt(\int_s^\tau \frac{1}{b(\sigma,s)}\,d\sigma \rt)d\tau\cr
&=-\frac{C_5}{C_1}b(t,s)\int_s^t \frac{1}{b(s,\sigma)}\,d\sigma + \frac{C_5}{C_1}\int_s^t \frac{b(\tau,s)}{b(\tau,s)}\,d\tau\cr
&\leq \frac{C_5}{C_1}(t-s).
\end{aligned}$$
Hence we have
\[
J(t,s) \leq \frac{C_5}{C_1}(1 + T),
\]
and this concludes the proof. 
\end{proof}

\begin{remark}
It is worthwhile to note that the only property of $u$ that is required for the proof of Theorem \ref{thm_ext} is its boundedness in the phase space. Indeed, the main idea behind the proof is to divide the particles into two groups: a group $A$ of particles colliding at $t_0$ and a group $B$ of particles that do not collide with particles from $A$. Then the interaction within group $A$ is singular and it outweighs any interaction within $B$ and any interaction between $A$ and $B$. Similarly, any effect of additional bounded forces is negligible compared to the singular interaction within $A$. Then, the interaction within $A$, after it outweighs all other influences, is used to prove the lack of collisions.
\end{remark}

%
%
%
%
\section{Flocking behavior}\label{flocking}
In this section, we provide a rigorous flocking estimate for the system \eqref{main_eq}. The proof follows a similar idea to the one used for a regular communication weight $\psi$. In the regular case, we conclude that $\|v\|^2\to 0$ from the fact that $\|v\|^2$ is integrable on $[0,\infty)$. However we may do it only since we know that $\|v\|^2$ is sufficiently regular, which is ensured by the regularity of $\psi$. In the case of singular $\psi$ we need to put some additional effort into proving sufficient, uniform-in-time, regularity of $\|v\|^2$. We do it by showing that the derivative of $\|v\|^2$ is a sum of an integrable function and a bounded function, which implies that $\|v\|^2$ is a sum of an absolutely continuous function and of a Lipschitz continuous function. Therefore, $\|v\|^2$ is a uniformly continuous function and its integrability ensures that $\|v\|^2\to 0$ as $t\to\infty$.

\begin{theorem}\label{thm_1} Suppose that $E_0 < \infty$, $\alpha \geq 1$, and the initial data $x_0$ satisfy
\[
\min_{1 \leq i \neq j \leq n}r_{ij}(0) > 0.
\]
Then there exists a unique smooth solution to the system \eqref{main_eq}-\eqref{ini_main_eq}. Furthermore, we assume that one of the two following hypotheses holds:
\begin{itemize}
	\item[(i)] $\beta \leq 1$; 
	\item[(ii)] $\beta > 1$ and 
	\bq\label{apt}
	\sum_{i=2}^{n}\int_{|x_{i-1}^0 - x_{i}^0 - z_{i-1}|^2}^\infty \phi(r)\,dr > \frac{1}{2Mn}\sum_{i,j=1}^n|v_i^0 - v_j^0|^2.
	\eq
\end{itemize}
Then we have
\[
\sup_{0 \leq t \leq \infty}\max_{1 \leq i,j \leq n}r_{ij}(t) < \infty \quad \mbox{and} \quad \max_{1 \leq i,j \leq n}|v_i(t) - v_j(t)|\to 0 \quad \mbox{as} \quad t \to \infty.
\]
\end{theorem}
\begin{remark}For the flocking estimate, we only need conditions for $\phi$ like boundedness, positivity, and the above assumption \eqref{apt}, which is automatically satisfied if $\phi$ integrates to infinity on any interval $[c,\infty)$.
\end{remark}
\begin{proof}[Proof of Theorem \ref{thm_1}]
{\bf Uniform boundedenss of $r_{ij}$:} It follows from Theorem \ref{thm_ext} that
 \[
\min_{1 \leq i \neq j \leq n}r_{ij}(t) > 0 \quad \mbox{for} \quad t \geq 0 
\]
and from Lemma \ref{lem_energy} that 
\[
E_2(x(t)) \leq E_0, 
\]
i.e.,
\bq\label{est_000}
\sum_{i=2}^n\int_{|x_{i-1}^0 - x_i^0 - z_{i-1}|^2}^{|x_{i-1}(t) - x_i(t) - z_{i-1}|^2}\phi(r)\,dr \leq \frac{1}{2Mn}\sum_{i,j=1}^n|v_i^0 - v_j^0|^2\quad \mbox{for} \quad t \geq 0.
\eq
On the other hand, under our main assumptions, we can find some constant $d_M > |x_{i-1}^0 - x_i^0 - z_{i-1}|$ such that
\begin{equation}\label{new1}
\sum_{i=2}^n\int_{|x_{i-1}^0 - x_i^0 - z_{i-1}|^2}^{d_M^2}\phi(r)\,dr = \frac{1}{2Mn}\sum_{i,j=1}^n|v_i^0 - v_j^0|^2.
\end{equation}
This, together with \eqref{est_000} yields
\[
\sum_{i=2}^n\int_{|x_{i-1}(t) - x_i(t) - z_{i-1}|^2}^{d_M^2}\phi(r)\,dr \geq 0,
\]
thus, we get
\bq\label{est_00}
|x_{i-1}(t) - x_i(t) - z_{i-1}| \leq d_M \quad \mbox{for} \quad i=2,\dots,n.
\eq
This further implies
\[
|x_i - x_j| \leq \sum_{k=i}^{j-1}|z_k| + (j-i)d_M,
\]
for any $1 \leq i\leq j \leq n$. Hence, for $t \geq 0$, we have
\[
\max_{1 \leq i \neq j \leq n}r_{ij}(t) \leq n\lt(\max_{1 \leq k \leq n-1}|z_k| + d_M\rt) = :C_0.
\]
Subsequently, we obtain
\bq\label{lower_bdd}
\psi_m := \min_{s \in [0,C_0]}\psi(s) = \frac{1}{C_0^\alpha} > 0.
\eq
{\bf Time-asymptotic velocity alignment behavior:} It follows from Lemma \ref{lem_energy} and Remark \ref{rmk_energy} that
\begin{align*}
\frac{d}{dt} E(x(t),v(t)) = -D(x(t),v(t)),
\end{align*}
which implies that
\begin{align}\label{wola5}
\int_0^\infty D(x(t),v(t))\,dt \leq E_0,
\end{align}
due to $E\geq 0$. Furthermore, we obtain
\begin{align}\label{wola2}
-D(x,v)\leq -\frac{K\psi_m}{2n}\sum_{i,j=1}^n|v_i-v_j|^2 = -K\psi_m\|v\|^2,
\end{align}
thanks to \eqref{lower_bdd} and the zero momentum condition \eqref{zero}. Thus we get
\begin{align}\label{wola6}
K\psi_m\int_0^\infty \|v(t)\|^2dt \leq \int_0^\infty D(x(t),v(t))\,dt \leq E_0.
\end{align}
This together with the estimate in the proof of Lemma \ref{lem_energy} and \eqref{wola2} gives
\begin{align}\label{wola1}
\frac{1}{2}\frac{d}{dt}\|v\|^2 = -D(x,v) + MI_2 \leq -K\psi_m\|v\|^2 + MI_2.
\end{align}
On the other hand, $I_2$ can be estimated as 
\begin{align*}
|I_2| &\leq \left|\sum_{i=2}^n\phi(|x_{i-1} - x_i - z_{i-1}|^2)\lal x_{i-1} - x_i - z_{i-1}, v_{i-1} - v_i \ral\right|\nonumber\\
&\leq d_M\sum_{i=1}^n|v_{i-1}-v_i| \leq Cd_M\|v\|\leq Cd_M\sqrt{E(0)},
\end{align*}
for some $C>0$, due to \eqref{est_00} and the energy estimate. Thus $I_2$ is bounded on $[0,\infty)$. Now we come back to \eqref{wola1} to see that the derivative of $\|v\|^2$ is a sum of an integrable function $(-D)$ due to \eqref{wola5} and of a bounded function $MI_2$. Hence we have
\begin{align*}
\|v(t)\|^2 = \underbrace{2\int_0^t(-D(x(s),v(s)))\,ds}_{=:f_1} + \underbrace{2\int_0^tMI_2(s)\,ds}_{=:f_2} + \|v(0)\|^2,
\end{align*}
where $f_1$ is absolutely continuous and $f_2$ is Lipschitz continuous. Both absolutely continuous and Lipschitz continuous functions are uniformly continuous and thus $\|v\|^2$ is uniformly continuous. After recalling from \eqref{wola6} that $\|v\|^2$ is also integrable, we conclude that $\|v\|^2\to 0$ with $t\to\infty$.
\end{proof}

\begin{remark} It is clear from Theorem \ref{thm_1} that
\[
\max\lt\{ 0, |z_i| - d_M  \rt\} \leq |x_i - x_{i+1}| \leq |z_i| + d_M \quad \mbox{for} \quad i=1,\dots,n.
\]
Furthermore, we find 
\[
r_{ij}(t) \geq \lt|\sum_{k=i}^{j-1}z_k \rt| - (j-i)d_M \quad \mbox{for} \quad i < j,
\]
i.e.,
\[
\min_{1 \leq i < j \leq n}r_{ij}(t) \geq \min_{1 \leq i < j \leq n}\lt|\sum_{k=i}^{j-1}z_k \rt| - (n-1)d_M,
\]
for $t \geq 0$.
\end{remark}

\begin{remark}\label{e2const}
As a direct consequence of Theorem \ref{thm_1}, we have
\[
\left|\frac{d}{dt}E_2(x(t))\right| = MI_2(t)\leq MCd_M\|v(t)\|\to 0\quad \mbox{as} \quad t\to\infty.
\]
\end{remark}

%
%
%
%
\section{Pattern formation}\label{formation}
In this section we prove that if the particles do not collide asymptotically then they form a pattern induced by the control. We briefly discuss the problem with asymptotic collisions in Remark \ref{permut}.

We first provide an enhancement of Young's inequality that will be significantly used later for the spatial pattern formation estimate.
\begin{lemma}\label{tedious}
Let $a_1,\dots,a_{n-1}$ be a set of vectors in $\R^d$. Then 
\begin{align*}
	-\sum_{i=1}^{n-1}|a_i|^2 + \sum_{i=1}^{n-2}\lal a_i, a_{i+1}\ral \leq -\delta^n \sum_{i=1}^{n-1}|a_i|^2,
\end{align*}
where $\delta\in(0,1)$ is a sufficiently small number.
\end{lemma}
\begin{proof}
Let $\delta > 0$ be a small number to be specified later. We take $\epsilon_1 = 1-\delta$ and use Young's inequality with $\epsilon_1$ to obtain
\begin{align*}
\lal a_1,a_2\ral\leq (1-\delta)|a_1|^2 + \frac{1}{4(1-\delta)}|a_2|^2.
\end{align*}
We take $\epsilon_i = 1-\delta^i- 1/(4\epsilon_{i-1})$ for $i\geq 2$. Then it is easy to prove by induction that $(1+\delta^i)/2 <\epsilon_i<1-\delta^i$ provided that $\delta$ is sufficiently small (for example $\delta=1/4$). Thus, by the recursive definition of $\epsilon_i$, for all $i=1,...,n-1$, we find $0<\epsilon_i+ 1/(4\epsilon_{i-1})\leq 1-\delta^n$. Hence, by Young's inequality, we have
$$\begin{aligned}
\sum_{i=1}^{n-2}\lal a_i, a_{i+1}\ral &\leq \sum_{i=1}^{n-2}\epsilon_i|a_i|^2 +\frac{1}{4\epsilon_i}|a_{i+1}|^2 \cr
&= \epsilon_1|a_1|^2+\sum_{i=2}^{n-1}\lt(\epsilon_i + \frac{1}{4\epsilon_{i-1}}\rt)|a_i|^2\cr
&\leq (1-\delta^n) \sum_{i=1}^{n-1}|a_i|^2.
\end{aligned}$$
This provides the desired result.
\end{proof}
We are now in a position to state the asymptotic spatial pattern formation result.
\begin{proposition}\label{pat}
Suppose that the assumptions of Theorem \ref{thm_1} are satisfied. Assume further that
\begin{align}\label{noas}
	\liminf_{t\to\infty}|x_i(t)-x_j(t)|>0,
\end{align}
for all $i,j \in \{1,\dots,n\}$. Then there exists a limit $\lim_{t\to\infty}x(t)=:x^\infty$ satisfying
\begin{align}\label{tes}
	x_i^\infty=x_{i-1}^\infty-z_{i-1}\qquad\mbox{for all} \quad i=2,\dots,n.
\end{align}
\end{proposition}
\begin{proof}For notational simplicity, in the rest of the proof, we denote
\begin{align*}
\phi_i:=\phi(|x_i-x_{i+1}-z_i|^2).
\end{align*}
Observe that by assumption \eqref{noas}, there exists a minimal distance between particles $\rho>0$ on the time interval $[t_0,\infty)$ for sufficiently large $t_0$. Thus we have
\begin{align*}
r_{ij}(t)\geq\rho \quad \mbox{for all} \quad i,j=1,\dots,n,
\end{align*}
for all $t\in[t_0,\infty)$. Consequently, this implies
\begin{align*}
\psi(r_{ij})\leq C(\rho)\quad\mbox{in} \quad [t_0,\infty).
\end{align*}
Let us first show that
\begin{align}\label{cp0}
x_i-x_{i+1} \to z_i \quad \mbox{as} \quad t \to \infty, 
\end{align}
for all $i=1,\dots,n-1$. It follows from the equation for $v_i$ in \eqref{main_eq} that
$$\begin{aligned}
&\frac{d}{dt}\sum_{i=1}^{n-1}\phi_i \lal x_i-x_{i+1}-z_i, v_i-v_{i+1}\ral \cr
&\quad = 2\sum_{i=1}^{n-1}\phi_i'\lal x_i-x_{i+1}-z_i, v_i-v_{i+1}\ral^2  +\sum_{i=1}^{n-1}\phi_i|v_i-v_{i+1}|^2\cr
&\qquad + \sum_{i=1}^{n-1}\phi_i \lt\lal x_i-x_{i+1}-z_i, \frac{K}{n}\sum_{j=1}^n \psi(r_{ij})(v_j - v_i)- \frac{K}{n}\sum_{j=1}^n \psi(r_{(i+1)j})(v_j - v_{i+1})\rt\ral\\
&\qquad + M\sum_{i=1}^{n-1}\phi_i \lal x_i-x_{i+1}-z_i , u_i -u_{i+1}\ral.
\end{aligned}$$
We then estimate each summand on the right-hand side separately. Clearly, we get
\[
\sum_{i=1}^{n-1}\phi_i|v_i-v_{i+1}|^2\leq 2E_1.
\]
Furthermore, by using Young's inequality with $\epsilon > 0$, we get
\[
2\sum_{i=1}^{n-1}\phi_i'\lal x_i-x_{i+1}-z_i, v_i-v_{i+1}\ral^2  \leq \epsilon\frac{|\phi'|_\infty^2}{\phi(d_M)}\sum_{i=1}^{n-1}\phi_i|x_i-x_{i+1}-z_i|^2 + C(\epsilon)E_1
\]
and
$$\begin{aligned}
& \sum_{i=1}^{n-1}\phi_i \lt\lal x_i-x_{i+1}-z_i , \frac{K}{n}\sum_{j=1}^n \psi(r_{ij})(v_j - v_i)- \frac{K}{n}\sum_{j=1}^n \psi(r_{(i+1)j})(v_j - v_{i+1})\rt\ral\cr
&\quad \leq \epsilon\sum_{i=1}^{n-1}\phi_i|x_i-x_{i+1}-z_i|^2 + C(\rho,\epsilon)E_1,
\end{aligned}$$
for any $\epsilon>0$. We finally estimate the term including the control $u$ as 
\begin{align*}
&\sum_{i=1}^{n-1}\phi_i \lal x_i-x_{i+1}-z_i, u_i -u_{i+1}\ral\cr
&\ = -2\phi_1^2|x_1-x_2-z_1|^2 +\lal \phi_1(x_1-x_2-z_1), \phi_2(x_2-x_3-z_2)\ral\cr
&\quad +\sum_{i=2}^{n-2}\phi_i\lal x_i-x_{i+1}-z_i, \cr
&\hspace{1.5cm}
\phi_{i-1}(x_{i-1}-x_{i}-z_{i-1})-2\phi_{i}(x_{i}-x_{i+1}-z_{i}) + \phi_{i+1}(x_{i+1}-x_{i+2}-z_{i+1})\ral\cr
&\quad +\phi_{n-1}\lal x_{n-1}-x_{n}-z_{n-1}, \cr
&\hspace{1.5cm} \phi_{n-2}(x_{n-2}-x_{n-1}-z_{n-2})-2\phi_{n-1}(x_{n-1}-x_{n}-z_{n-1})\ral\cr
&\ =2\sum_{i=1}^{n-2}\lal\phi_i(x_i-x_{i+1}-z_i),\phi_{i+1}(x_{i+1}-x_{i+2}-z_{i+1})\ral -2\sum_{i=1}^{n-1}\phi_i^2|x_i-x_{i+1}-z_i|^2\cr
&\ =:{\mathcal L}.
\end{align*}
By Lemma \ref{tedious}, there exists a positive constant $\delta>0$ such that
\begin{align*}
M{\mathcal L}\leq -2M\delta^n\sum_{i=1}^{n-1}\phi_i^2|x_i-x_{i+1}-z_i|^2\leq -\frac{2M}{\phi(d_M)}\delta^n\sum_{i=1}^{n-1}\phi_i|x_i-x_{i+1}-z_i|^2.
\end{align*}
Combining the above estimates with suitably chosen $\epsilon>0$, we end up with
\bq\label{cp1}
\frac{d}{dt}\sum_{i=1}^{n-1}\phi_i\lal x_i-x_{i+1}-z_i, v_i-v_{i+1}\ral\leq C(\rho)E_1-\frac{M}{\phi(d_M)}\delta^n\sum_{i=1}^{n-1}\phi_i|x_i-x_{i+1}-z_i|^2.
\eq
By \eqref{wola6}, the kinetic energy $E_1$ is integrable and by the following inequality
\begin{align*}
\sum_{i=1}^{n-1}\phi_i|\lal x_i-x_{i+1}-z_i, v_i-v_{i+1}\ral|\leq Cd_M\sqrt{E_0},
\end{align*}
for some $C>0$ independent of $t$, the left-hand side of \eqref{cp1} is a derivative of a bounded function. This implies that $\sum_{i=1}^{n-1}\phi_i|x_i-x_{i+1}-z_i|^2$ is also integrable in $[t_0,\infty)$, and thus we conclude that $\sum_{i=1}^{n-1}|x_i-x_{i+1}-z_i|^2\to 0$ as $t\to\infty$, since it is Lipschitz continuous. Hence, \eqref{cp0} is proved. This, together with the fact that we fixed $0$ as the center of mass (recall \eqref{zero}) and the help of basic linear algebra implies that $x_i(t)$ is convergent with $t\to\infty$, and its limit $x_i^\infty$ satisfies \eqref{tes} for all $i=1,\cdots, n$. This completes the proof. 
\end{proof}

\begin{remark}\label{permut}
It appears that the impossibility of collisions between particles, ensured by Theorem \ref{thm_ext}, plays a role in pattern formation and is the reason for the need of assumption \eqref{noas}. Indeed, if we consider the simplest case of two particles on a line, with $z_1=-1$, then, the resulting pattern has to be of the form $x_2^\infty = x_1^\infty+1>x_1^\infty$. However, if initially $x_1(0)>x_2(0)$, at some point the particles change order, and thus collide, which is impossible due to Theorem \ref{thm_ext}. It is also clear intuitively: the particles are forbidden from colliding and if the control would result in a collision, the singularity of the communication weight $\psi$ prevails and the pattern cannot be formed. We numerically investigate this issue, see Figure \ref{fig:lines} below. The one-dimensional case is special in the sense that the collisions are unavoidable if the order of the particles has to be changed. It is however a much more complex question in $d\geq 2$. See Figure \ref{fig:degenerate2d} and the corresponding discussion for a particular 2 dimensional case. 
\end{remark}
\textcolor{black}{
A full discussion of the {\it a priori} assumption \eqref{noas} is outside the scope of this paper. In fact, as mentioned above, Theorem \ref{thm_ext} implies that there is no collision between particles for all time when $\alpha \geq 1$, thus the assumption \eqref{noas} excludes a possible collision at $t = \infty$. On the other hand, based on the argument in Theorem \ref{thm_1}, we provide a class of initial data leading to a global-in-time minimal distance, which is strictly positive, between the particles, and thus by Proposition \ref{pat}, to the desired pattern formation. Such a class of initial data can be roughly described as follows:
}

\textcolor{black}{
{\bf (A)} {\it 
If the initial energy $E_0$ and all $|x^0_{i-1}-x_i^0-z_{i-1}|$, for $i$ between $2$ and $n$, are sufficiently small compared to $\min_{i<j}\lt|\sum_{k=i}^{j-1}z_k \rt|$, then 
there exists a positive global-in-time minimal distance between the particles and the pattern forms.
}
}

We provide the above statement explicitly in the case of $\beta\in(0,1)$.
\begin{corollary}
Let $\beta\in(0,1)$. Under the assumptions of Theorem \ref{thm_1} with the initial data satisfying for all pairs $1\leq i<j\leq n$
\begin{align}\label{v4asum}
\lt|\sum_{k=i}^{j-1}z_k \rt|^2 &> \left(\frac{j-i}{j+1-i}\right)^{2-\frac{1}{1-\beta}}(j+1-i)\\
&\quad \times\left(\left((1-\beta)C_0^* + 1\right)^\frac{1}{1-\beta} - 1+ \sum_{k=i+1}^{j}|x^0_{k-1}-x^0_k -z_{k-1}|^2\right)\nonumber
\end{align}
with 
\begin{align*}
C_0^* := \frac{1}{2Mn}\sum_{i,j=1}^n|v_i^0-v_j^0|^2,
\end{align*}
there exists a global-in-time minimal distance between the particles. Therefore the particles converge asymptotically to a pattern as indicated in Proposition \ref{pat}.
\end{corollary}
\begin{proof}
Applying the triangle inequality to \eqref{est_00} we deduce that
\[
r_{ij}(t) \geq \lt|\sum_{k=i}^{j-1}z_k \rt| - (j-i)d_M \quad \mbox{for} \quad i < j,
\]
and $t \geq 0$. It remains to fix $1\leq i<j\leq n$ and prove that
\begin{align}\label{new2}
\lt|\sum_{k=i}^{j-1}z_k \rt|^2 > (j-i)^2d_M^2.
\end{align}
The starting point is the definition of $d_M$ in \eqref{new1}, which implies that
\begin{align*}
(j-i)\Phi(d_M^2) \leq C_0^* + \sum_{k=i+1}^{j}\Phi(|x^0_{k-1}-x^0_k -z_{k-1}|^2),
\end{align*}
where $\Phi$ is the anti-derivative of $\phi$ with $\Phi(0)=0$, i.e. $\Phi(s) = \frac{1}{1-\beta}(1+s)^{1-\beta} - \frac{1}{1-\beta}$.
Note that the anti-derivative $\Phi$ is positive, increasing and concave on $[0,\infty)$. Moreover, $C_0^*\geq 0$ belongs to the image $\Phi([0,\infty))$. Set $C_0^{**}:= \Phi^{-1}(C_0^*)$. Then, by concavity of $\Phi$, we have
\begin{align}\label{v41}
\Phi\left(\left(\frac{j-i}{j+1-i}\right)^\frac{1}{1-\beta}d_M^2\right)
&\leq\frac{j-i}{j+1-i}\Phi(d_M^2) \\
&\leq \frac{\Phi(C_0^{**}) + \sum_{k=i+1}^{j}\Phi(|x^0_{k-1}-x^0_k -z_{k-1}|^2)}{j+1-i}\nonumber\\
&\leq \Phi\left(\frac{C_0^{**}+ \sum_{k=i+1}^{j}|x^0_{k-1}-x^0_k -z_{k-1}|^2}{j+1-i}\right).\nonumber
\end{align}
To see the left-most inequality in \eqref{v41}, let us consider the function
\begin{align*}
f_c(s) = c\Phi(s) - \Phi(c^\frac{1}{1-\beta}s)\qquad\mbox{for}\ s\geq 0 \ \mbox{and}\ c\in(0,1).
\end{align*}
Clearly, $f_c(0)=0$ and
\begin{align*}
f'_c(s) = c(1+s)^{-\beta} - c^\frac{1}{1-\beta}(1+c^\frac{1}{1-\beta}s)^{-\beta}.
\end{align*}
Moreover, $f'_c\geq 0$ is equivalent to
\begin{align*}
1+s\leq c^{\frac{1}{\beta}-\frac{1}{(1-\beta)\beta}}(1+c^\frac{1}{1-\beta}s),
\end{align*}
which is the case, since for $\beta\in(0,1)$, we have 
\[
\frac{1}{\beta}-\frac{1}{(1-\beta)\beta}<0 \quad \mbox{and}  \quad \frac{1}{\beta}-\frac{1}{(1-\beta)\beta} + \frac{1}{1-\beta} = 0.
\]
Thus, by monotonicity of $\Phi$ in \eqref{v41}, we have
\begin{align*}
\left(\frac{j-i}{j+1-i}\right)^\frac{1}{1-\beta}d_M^2\leq \frac{C_0^{**}+ \sum_{k=i+1}^{j}|x^0_{k-1}-x^0_k -z_{k-1}|^2}{j+1-i},
\end{align*}
which implies that \eqref{new2} holds, provided that
\begin{align*}
\lt|\sum_{k=i}^{j-1}z_k \rt|^2 > \left(\frac{j-i}{j+1-i}\right)^{2-\frac{1}{1-\beta}}(j+1-i)\Bigg(C_0^{**}+ \sum_{k=i+1}^{j}|x^0_{k-1}-x^0_k -z_{k-1}|^2\Bigg).
\end{align*}
Recalling that $C_0^{**} = \Phi^{-1}(C_0^*)$ we finish the proof.
\end{proof}
\begin{remark}
When $\beta\geq 1$, then the general statement (A) still holds, but the explicit assumption, similar to \eqref{v4asum} becomes significantly more convoluted. The main problem is that the left-most inequality in \eqref{v41} fails and needs to be compensated. However, a direct calculation leads to a condition similar to \eqref{v4asum} in the spirit of (A).
\end{remark}

%
%
%
%

\section{Numerical experiments}\label{numerics}

\begin{figure}[!t]\begin{center}
	\includegraphics[width=\columnwidth]{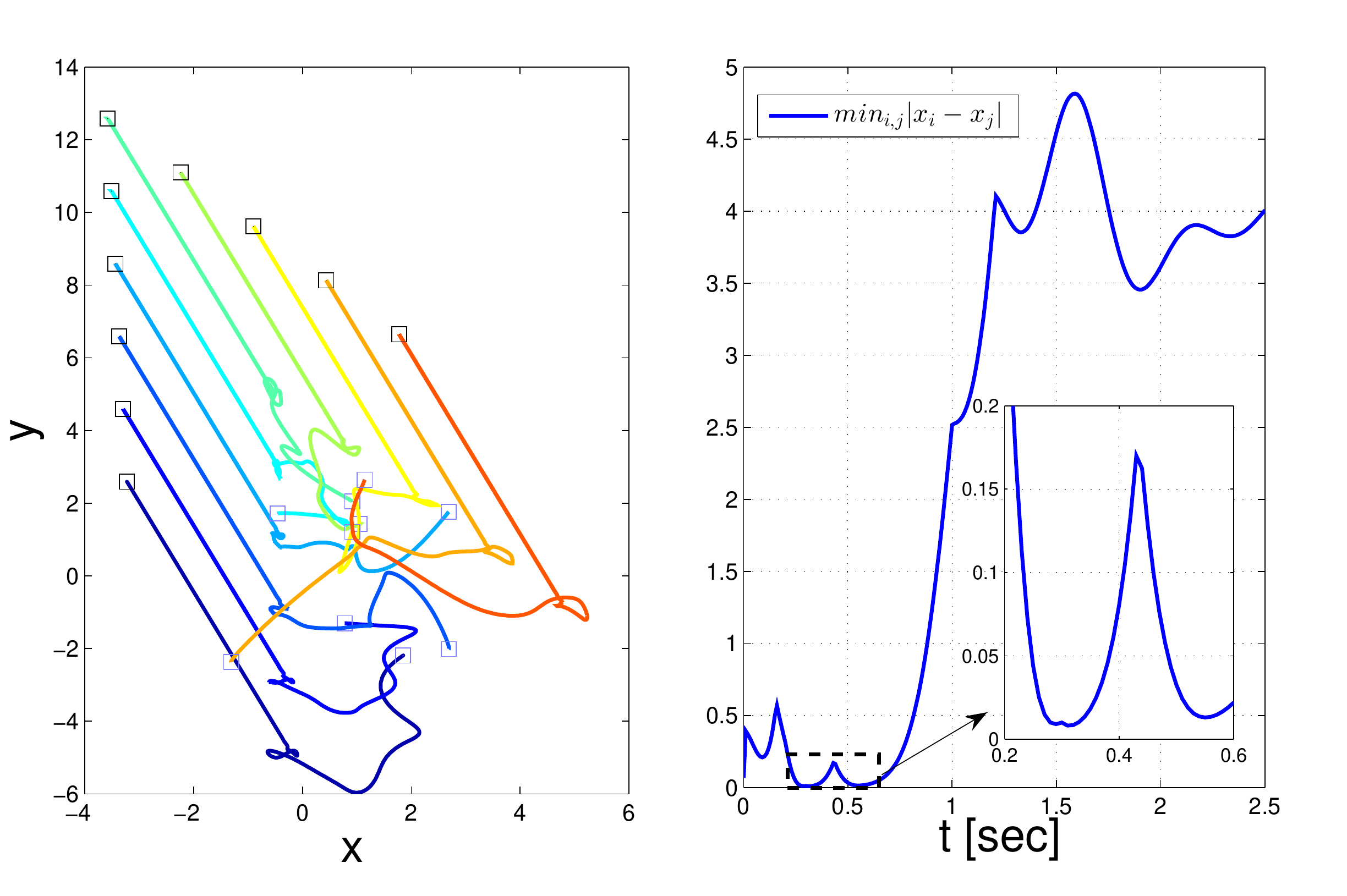}\\
	\caption{Trajectories and collision detection for a bird-like pattern} \label{fig:bird_trajectories}
\end{center}
\end{figure} 

\begin{figure}[t]\begin{center}
	\includegraphics[width=\columnwidth]{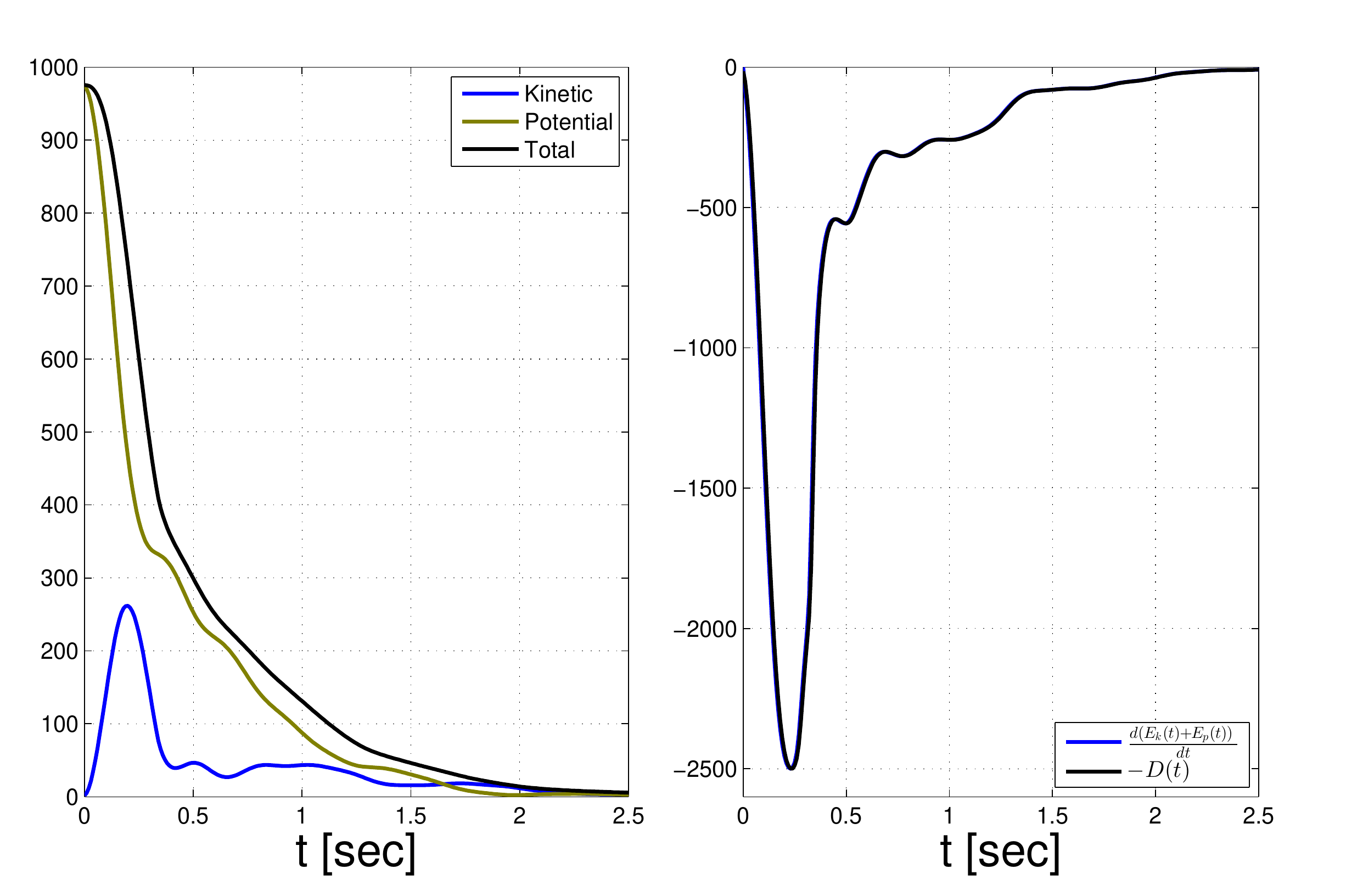}\\
	\caption{Energies (left) and dissipation (right) for a bird-like pattern} \label{fig:bird_energies}
\end{center}
\end{figure} 

%
%
%
%
%
In this section we present simulations for planar and spatial configurations with the studied model.

\paragraph{\textbf{Bird-like pattern}} We consider a bird-like flocking pattern in the 2 dimensional space with $n=10$ agents. The parameters are chosen as $K=10, M=50, \alpha=1.1$ and $\beta=0.5$. If we consider the average velocity of the flock, which remains constant,
$$v_c(t)=\frac 1n \sum_{i=1}^nv_i(t)=v_c(0)=\begin{pmatrix}
v_x \\ v_y
\end{pmatrix},$$ 
we define the angle $\theta $ of travel with respect to the $x$ axis, that is, $\theta=\arctan(v_y/v_x)$. If we chose the desired inter-particle spacings, $z_i$, to be 
$$z_i=\begin{cases} -2\left(\cos\left(\theta-\dfrac{\pi}{9}\right),\sin\left(\theta-\dfrac{\pi}{9}\right)\right) & i\leq\left\lfloor\dfrac{n}{2}\right\rfloor \vspace{2mm}\\ 
2\left(\cos\left(\theta-\dfrac{\pi}{9}\right),\sin\left(\theta-\dfrac{\pi}{9}\right)\right) & i>\left\lfloor\dfrac{n}{2}\right\rfloor \end{cases}, $$  
the control should achieve a bird-like pattern in steady state. Figure \ref{fig:bird_trajectories} left shows the trajectories followed by the agents on the $x-y$ plane. Figure \ref{fig:bird_trajectories} right shows the plot of $\min_{i,j}|x_i(t)-x_j(t)|$. A few particles are initially very close to each other, but the plot, and its zoomed in view, reveal that no collisions occur. In Figure \ref{fig:bird_energies} we show the energy decomposition of the system and its dissipation. It can be noted that the second part of Lemma \ref{lem_energy} is satisfied.

\paragraph{\textbf{A single circle}} 
Now, we consider a set of inter-particle alignments $z_i$ in the control term that achieves a circular formation pattern with a single agent at the center. Figure \ref{fig:circle} left shows the trajectories for $n=50$ agents on the plane, with the model parameters as in the previous example. The initial conditions are such that some agents are close to each other and about to collide. Figure \ref{fig:circle} right shows the plot of $\min_{i,j}|x_i(t)-x_j(t)|$ over time, illustrating that the singular influence term causes collision-avoidance. However, there is at least one pair of agents involved in near misses in two opportunities. 

\begin{figure}[!h]\begin{center}
	\includegraphics[width=\columnwidth]{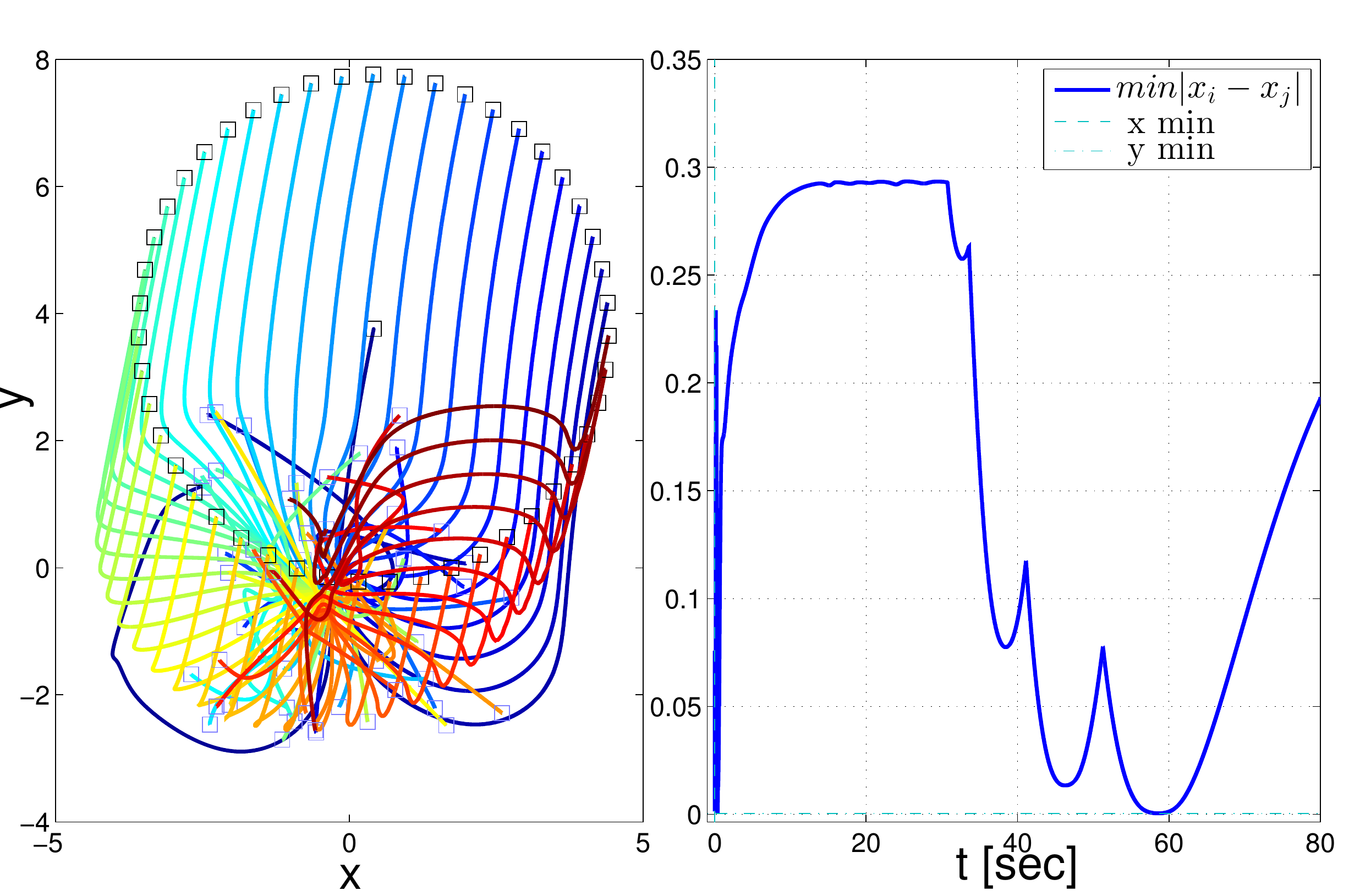}\\
	\caption{Trajectories and collision detection for a circle formation with an agent at the center} \label{fig:circle}
\end{center}

\end{figure} 

\paragraph{\textbf{Collisions in finite time}} Here we consider an example to illustrate that, even when collisions are expected from the initial conditions and desired final formation, the singular influence prevails and no collision occurs in finite time when $\alpha \geq 1$. In particular, we consider 4 particles in the 1 dimensional space. Their initial positions and desired final formations are such that in the steady state they must crossover (and therefore collide). In the following we consider 4 cases: Regular weight $\psi(r)=(1+r)^{-\alpha}$ and singular weight $\psi(r)=r^{-\alpha}$, both with $\alpha=\{0.5,1.5\}$. 

Figure \ref{fig:lines} shows the position of the 4 particles over time for the aforementioned cases. In particular we chose every parameter as before, except for $\alpha$, and the initial conditions $x_i(0)=0.5i$ for $i=1,2,3$, $x_4(0)=-1$, $v_i(0)=(-1)^ii/4$ for $i=1,2,3$ and $v_4(0)=1.$ The desired formation is given by the selection $z_i=-2$ for all $i$, which should put the agent with position $x_4(t)$ ahead of the rest, considering the initial conditions. It can be noted that the only case where the agents do not collide is for the singular weight with $\alpha=1.5$. In that case, the particles are collapsing together but never really colliding, since Theorem \ref{thm_1} prohibits collisions in finite time.

\begin{figure}[!h]\begin{center}
		\includegraphics[width=\columnwidth]{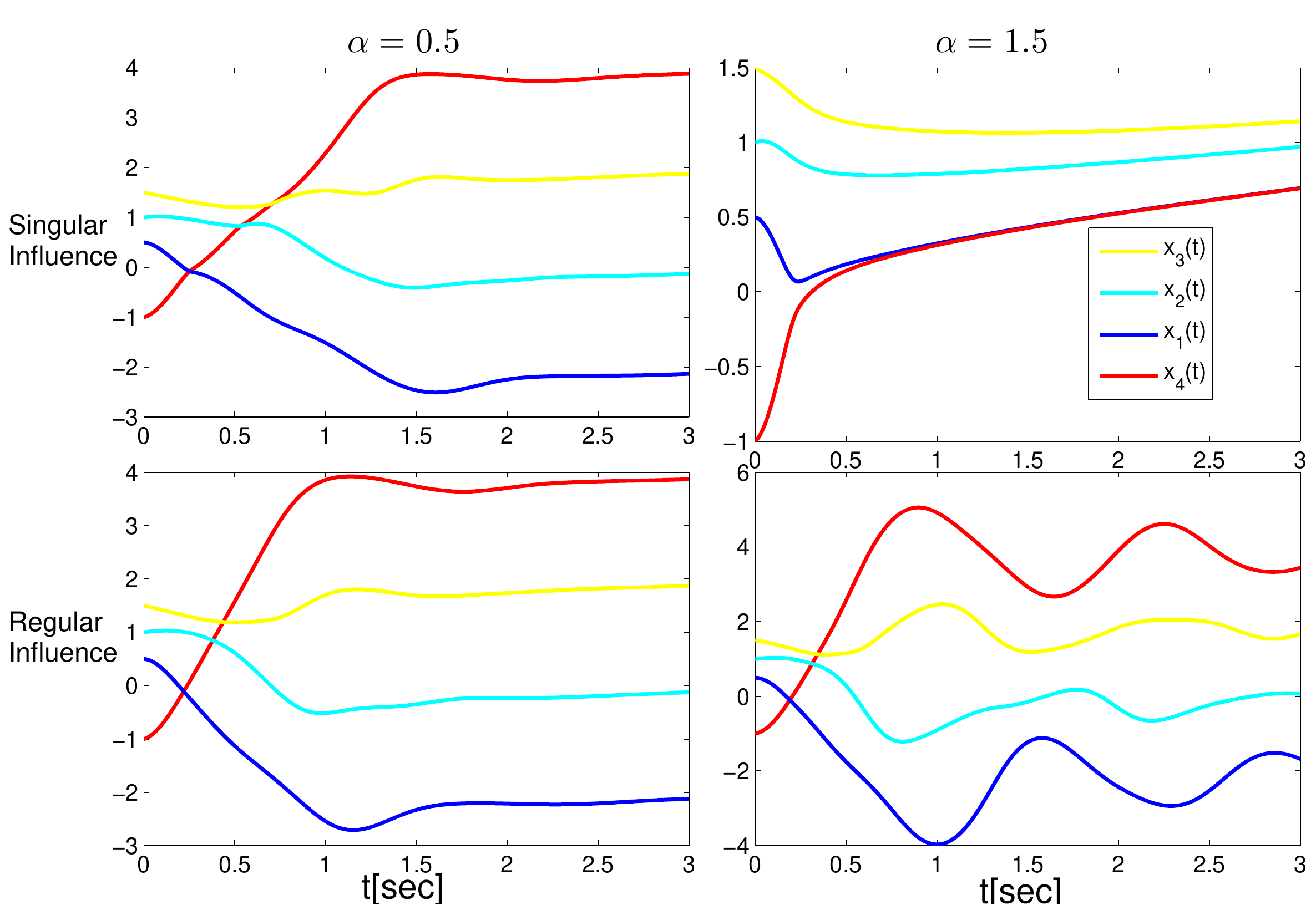}\\
		\caption{Positions over time of 4 particles on a line} \label{fig:lines}
	\end{center}
\end{figure}

We provide animations of these cases at \href{https://tinyurl.com/yapwy924}{https://tinyurl.com/yapwy924}. \paragraph{\textbf{Degenerate Cases}} The previous 1D case can be considered degenerate and it will be impossible to have a successful pattern acquisition, given the initial conditions when $\alpha\geq 1$. In higher dimensions, as mentioned in Remark \ref{permut}, the situation is harder to assess. For $\beta=0.5$, $\alpha=1.1$, $K=60$ and $M=50$, we consider 4 agents in the 2 dimensional space with initial positions given by $x_1(0)=(-1,0)$, $x_2(0)=(0,1)$, $x_3(0)=(1,0)$, and $x_4(0)=(0,-1)$, that is, the agents are positioned at the vertices of a square. We will consider that the agents are initially at rest and the desired spatial formation is encoded in the selections $z_1=(1,1)$, $z_2=(1,-1)$ and $z_3=(-1,-1)$. The control will then try to switch the agents in a diagonal fashion, since $v_i=(0,0)$ for all $i$, and we expect that the agents will end up in the positions $x_1^\infty=(1,0)$, $x_2^\infty=(0,-1)$, $x_3^\infty=(-1,0)$, and $x_4^\infty=(0,1)$. However, as seen in Figure \ref{fig:degenerate2d} left, the agents collide in pairs, at the middle points of two sides of the square. Switching the initial positions of agents 3 and 4, and again forcing them to swap positions diagonally, we observe that the agents do reach the desired formation.

\begin{figure}[!t]\begin{center}
		\includegraphics[width=\columnwidth]{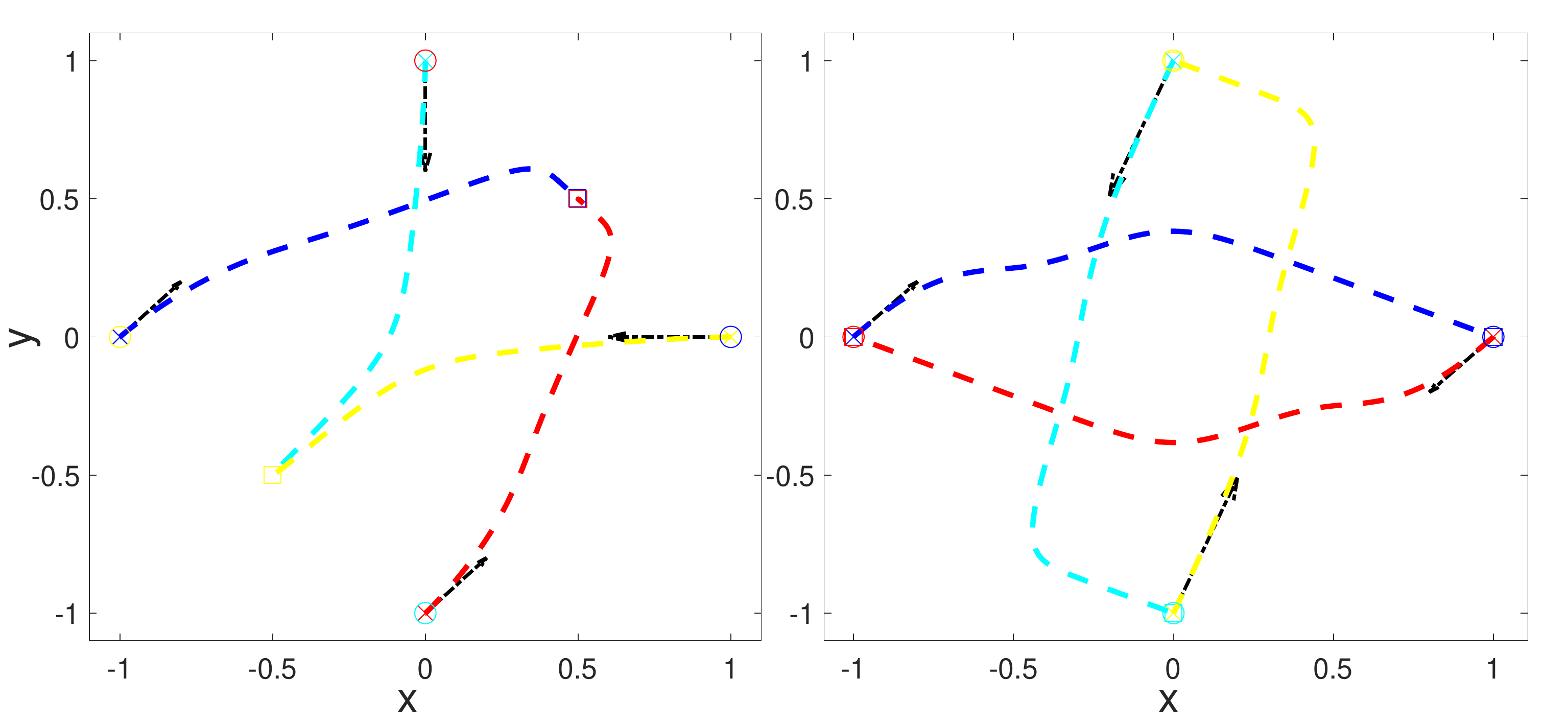}\\
		\caption{4 Particles starting on the vertices of a square and their trajectories. Crosses: Initial particle positions. Circles: Desired final pattern. Squares: Final particle positions. Black dashed vectors: Initial accelerations of the particles.} \label{fig:degenerate2d}
	\end{center}
\end{figure} 
	
The key difference resides in the symmetry of the initial accelerations, which in this case are completely determined by the control (as every agent starts at rest). It can be easily checked that for the first case $u_1(0)=u_4(0)=-0.5(u_2(0)+u_3(0))\neq \vec{0}$ and given the symmetry of the initial conditions, the controls and the influence term, the vectors $d(v_1(t)+v_4(t))/dt$ and $d(v_2(t)+v_3(t))/dt$ will remain in the same direction for all $t$. We can observe in Figure \ref{fig:degenerate2d} that, due to the symmetry of the acceleration terms and initial conditions, $r_{12}(t)=r_{34}(t)$ and $r_{13}(t)=r_{24}(t)$ for all $t$. Additionally, the symmetry also suggests that $x_1(t)-x_2(t)-z_1=x_3(t)-x_4(t)-z_3$ for all $t$. With this, the center of mass of the pair of agents with indexes 1 and 4 will stay on the line that crosses the origin with slope 1. The only option for the agents to try and reach the desired pattern is for them to collide in pairs somewhere on this line.

On the contrary, in the second case $u_1(0)+u_4(0)=-(u_2(0)+u_3(0))=\vec{0}$ but all initial accelerations are different. The centre of mass of agents 1 and 4 travels along the $x$ axis (and the centre of mass of agents 2 and 3 travels along the $y$ axis), but since the initial accelerations of these agents force them to different hemispheres, they can switch positions. The only possibility for collisions would be among particles that are not switching places, which is avoided by the lack of symmetry in the initial accelerations.

These particular cases are easily avoided by small perturbations on the initial conditions, or with selections of the control that take into account the initial configuration of the agents. In applications, where sensors and actuators are imperfect, it is unlikely that these exceptions could occur. However, if velocities are reaching consensus but the agents are still not close to the desired formation, the values $z_i$ could be perturbed slightly in orthogonal or random directions, in order to avoid cases such as this example. In a similar vein, the presence of more agents or higher dimensions would also tend to decrease the likeliness of such degenerate cases. Considering the several different cases that we have experimented with, it seems that condition \ref{noas} is satisfied most of the time.

\paragraph{\textbf{The Olympic Rings at PyeongChang 2018}} In Figure \ref{fig:snaprings} we present time snapshots of the trajectories followed by a system of 50 agents in the 3 dimensional space. The parameters are chosen as before but with the sequence $z_i$ selected to obtain a final spatial pattern that describes the Olympic Rings. At each snapshot we plot the position of every agent as points and their instantaneous velocities as vectors. Initially, the agents are located at random positions and satisfying $v_c(0)=\vec{0}\in\mathbb{R}^3$. Moreover, we plot the desired final pattern, the Olympic Rings formation, as empty circles on the plane with the third coordinate equal to $0$.

It is possible to observe for some agents that at $t=0.5[s]$ the magnitudes of their velocities are greater than their initial ones. This is not, however, inconsistent with $dv_c(t)/dt=0$. For $t=5[s]$ the agents are approaching the plane where the final formation resides in, although the projection of all of the agents positions on this plane is not the desired formation yet. At $t=200[s]$ we observe the desired pattern with the velocities of the agents vanished almost completely. We invite the reader to watch a full animation of the trajectories (with some time scaling for brevity) at \href{https://youtu.be/C7UDGRudsyA}{https://youtu.be/C7UDGRudsyA}.

\begin{figure*}\begin{center}
	\includegraphics[width=\textwidth]{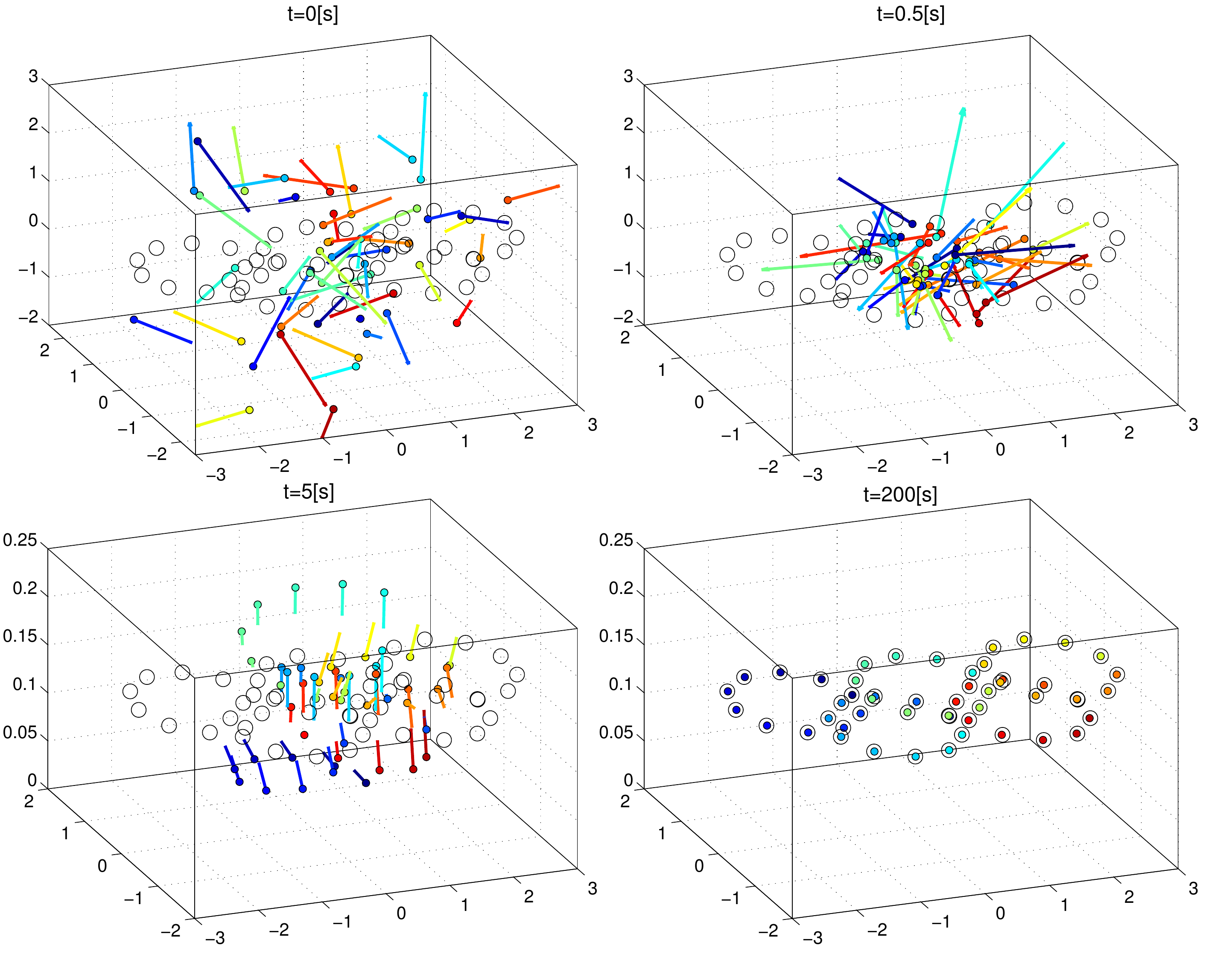}\\
	\caption{Snapshots at different times for an \emph{Olympic rings} pattern} \label{fig:snaprings}
\end{center}
\end{figure*}

%
%
%
%
\section*{Conclusions and future work}
We have presented a decentralized feedback law for formation control over a singular Cucker-Smale type model. By using a reference vector between agents, we are able to enforce pattern formations by means of decentralized control action. An energy analysis shows the collision-avoidance feature of the closed-loop dynamics in dimensions higher than 1. There are different extensions of the present work we are currently addressing: the design of collision-avoidance control laws for the non-singular Cucker-Smale model, the study of alternative decentralization strategies (e.g. by sparsification of a centralized controller), the stabilization towards time-dependent reference trajectories, and the understanding of formation control at a mean-field level. 
%
%
%
%
\section*{Acknowledgements}
YPC was supported by NRF grant (2017R1C1B2012918 and 2017R1A4A1014735) and POSCO Science Fellowship of POSCO TJ Park Foundation. APR was supported by the Advanced Center for Electrical and Electronic Engineering, Basal Project FB0008, and by the Grant FONDECYT 3160738, CONICYT Chile. JP was supported by the Polish MNiSW grant Mobilno\' s\' c Plus no. 1617/MOB/V/2017/0 and by the NSF grant RNMS11-07444 (KI-Net). DK and AP wish to dedicate this paper to the memory of Mario Salgado Brocal.

\bibliographystyle{siamplain}
\bibliography{bibckpp}

\end{document}